\documentclass[oneside, 11pt]{amsart}
\usepackage[utf8]{inputenc}

\usepackage{amsfonts}
\usepackage{amsmath}
\usepackage{amssymb}
\usepackage{amsthm}
\usepackage[a4paper]{geometry}
\usepackage{mathrsfs} %for nicer caligraphic letters via mathscr
\usepackage[dvipsnames]{xcolor}
\usepackage[pagebackref=false,colorlinks=true, linkcolor=blue, citecolor=blue]{hyperref}

\usepackage{appendix}
\usepackage{enumerate}
\usepackage{geometry}
\usepackage{tikz-cd} 
\usepackage{orcidlink}

\theoremstyle{plain}
\newtheorem{theorem}{Theorem}
\newtheorem{proposition}[theorem]{Proposition}
\newtheorem{lemma}[theorem]{Lemma}
\newtheorem{corollary}[theorem]{Corollary}

\theoremstyle{definition}

\numberwithin{theorem}{section}
\numberwithin{equation}{section} %Equations are numbered as #section.#equation

\newcommand{\set}[1]{\left\{ #1 \right\}}
\newcommand{\norm}[1]{\left \lVert  #1 \right \rVert}
\newcommand{\abs}[1]{\left\lvert #1 \right\rvert}

\newcommand{\vertiii}[1]{{\left\vert\kern-0.25ex\left\vert\kern-0.25ex\left\vert #1 
    \right\vert\kern-0.25ex\right\vert\kern-0.25ex\right\vert}}

\newcommand{\Z}{\mathbb{Z}}

\newcommand{\R}{\mathbb{R}}
\newcommand{\N}{\mathbb{N}}

\newcommand{\calF}{\mathcal{F}}

\newcommand{\calM}{\mathcal{M}}

\title[Time-periodic behaviour without reversibility]{On the absence of time-translation symmetry breaking in some non-reversible interacting particle systems}
\author{Jonas K\"oppl \ \orcidlink{0000-0001-9188-1883}}
\address{Weierstrass Institute, Berlin, Germany. }
\email{koeppl@wias-berlin.de}
\date{\today}
\begin{document}

\begin{abstract} 
    The conditions under which stochastic systems of infinitely many interacting particles can maintain sufficient spatial order to move coherently along a time-periodic orbit, thereby breaking the time-translation invariance of the underlying dynamical equation, have been an elusive issue. 
    Via a free energy technique, we prove that if a non-reversible interacting particle system on $\Z^d$, $d=1,2$, with strictly positive rates admits a product measure as a stationary measure, then it cannot exhibit time-periodic behaviour. This provides a first step towards a general conjecture that time-periodic behaviour cannot occur in one- and two-dimensional systems with short-range interactions and constitutes the first result for non-reversible dynamics in dimension two. 
\end{abstract}
\maketitle
\section{Introduction}
We consider interacting particle systems on the $d$-dimensional integer lattice $\mathbb{Z}^d$, which are Markov processes on the state space $\Omega = \{1,\dots, q\}^{\Z^d}$ specified in terms of generators of the form 
\begin{equation*}
    \mathscr{L}f(\eta) = \sum_{x \in \Z^d}\sum_{\xi_x} c_x(\eta,\xi_x)\left[f(\xi_x \eta_{x^c})-f(\eta)\right], \quad \eta \in \Omega, 
\end{equation*}
for local functions $f:\Omega \to \R$. The transition rates $c_x(\eta,\xi_x)$ can be interpreted as the infinitesimal rate at which the particle at site $x$ switches from the state $\eta_x$ to $\xi_x$, given that the rest of the system is currently in state $\eta_{x^c}$.  We will denote the associated semigroup by $(P_t)_{t \geq 0}$ and the set of probability measures on $\Omega$ by $\calM_1(\Omega)$. 
The question we are primarily interested in is which symmetries of the transition rates are inherited by the stationary behaviour of the dynamics. \medskip

As an example, consider the Glauber dynamics for the Ising model. Then it is well known that the transition rates are both invariant under a global spin flip and lattice translations. However, in $d\geq 2$ at sufficiently low temperatures there exist time-stationary measures which are not invariant under global spin flips and in $d\geq3$ there exist non-translation-invariant time-stationary measures. So these two symmetries can be spontaneously broken in infinite volume systems. 
Another obvious and hence often overlooked symmetry is that the transition rates do not depend on time, i.e., the generator is autonomous, and thus invariant under time shifts. Of course, trivially any time-stationary measure is also invariant under time shifts and one needs to be a bit more careful when defining the notion of time-translation symmetry breaking. \medskip 

\noindent We will say that \textit{(spontaneous) time-translation symmetry breaking} occurs if 
\begin{enumerate}[\bfseries (TTSB)]
    \item
        $\exists \mu_0 \in \calM_1(\Omega) \ \exists \tau>0$ such that $\mu P_\tau = \mu_0$ and for all $s \in (0,\tau)$ we have $\mu P_s \neq \mu_0$. 
\end{enumerate}
In other words, there exists a non-trivial time-periodic orbit $(\mu P_s)_{s\in[0,\tau]}$ in the space $\mathcal{M}_1(\Omega)$ of probability measures on $\Omega$. This means that the usual continuous shift symmetry by any $t >0$ is reduced to a discrete symmetry by $t \in \tau\mathbb{N}$. 
\medskip 

First note, that if an interacting particle system with strictly positive transition rates can break the time-translation symmetry  then, as a consequence of the ergodic theorem for continuous-time Markov chains on finite state spaces, this can only happen in infinite volumes, so it is necessarily a many-body phenomenon that arises due to interactions. The conditions under which such noisy systems can maintain sufficient spatial order to move coherently along a time-periodic orbit, thereby spontaneously breaking the time-translation invariance of the underlying dynamical equation, have been a contentious issue in the physics literature \cite{grinstein_temporally_1993, bennett_stability_1990, chate_collective_1992}, and a mathematical treatment of such broken-symmetry, macroscopically time-dependent states has been mostly limited to mean-field like approximations \cite{collet_rhythmic_2016}.

As it turns out, the possibility or non-possibility of time-translation symmetry breaking appears to heavily depend on the dimension $d$. Non-rigorous arguments \cite{grinstein_temporally_1993} and extensive numerical studies of specific models \cite{avni_nonreciprocal_2025,avni_dynamical_2025, guislain_collective_2024} seem to suggest that interacting particle systems with short-range interactions can exhibit $\mathbf{(TTSB)}$ in dimensions $d\geq3$, but cannot produce stable time-periodic behaviour  in dimensions $d=1,2$. 

Mathematically, the literature regarding this issue is much scarcer. 
While a classical result by Mountford and Ramirez--Varadhan, see \cite{mountford_coupling_1995} and \cite{ramirez_relative_1996}, shows that no such system with finite-range transition rates can exist in one spatial dimension, recent constructions show that systems with long-range interactions can indeed exhibit $\mathbf{(TTSB)}$, even in dimension one and two, see \cite{jahnel_time-periodic_2025} for $d=1,2$ and  \cite{jahnel_class_2014} for $d \geq 3$. 

Recently, in \cite{jahnel_long-time_2025}, it was additionally shown that in one and two dimensions, any \textit{reversible} short-range system cannot exhibit $\mathbf{(TTSB)}$. 

Since the examples in \cite{jahnel_time-periodic_2025} heavily rely on the long-range nature of the interactions, the current theoretical evidence leads to the same conjecture as in the physics literature: $\mathbf{(TTSB)}$ cannot occur in one- and two-dimensional systems with short-range interactions, even in non-reversible systems. 

In this article, we want to take a first step towards verifying this conjecture and investigate the possible long-time behaviour of interacting particle systems in one and two spatial dimensions \textit{without} the additional assumption of reversibility, at least for a specific class of systems. For this, we combine the proof strategy of \cite{jahnel_long-time_2025} together with key ideas from \cite{jahnel_dynamical_2023} and \cite{ramirez_uniqueness_2002} that allow us to go beyond reversible systems. 
\subsection*{Organisation of the manuscript} In Section \ref{section:setting-and-main-results}, we introduce the precise framework in which we are working and setup the required notation before we state our main results. We then comment on possible extensions to more general non-reversible systems in dimensions $d=1,2$ in Section \ref{section:outlook}. After this we finally start with the main work and provide an outline of the proof strategy in Section \ref{section:proof-strategy} before diving into the details in Section \ref{section:proof-relative-entropy-loss-principle}.

\section{Setting and main results}\label{section:setting-and-main-results}
Let $q \in \N$ and consider the configuration space $\Omega := \Omega_0^{\Z^d} = \{1,\dots,q\}^{\Z^d}$, which we will equip with the usual product topology and the corresponding Borel sigma-algebra $\calF$. For $\Lambda \subset \Z^d$ let $\calF_\Lambda$ be the sub-sigma-algebra of $\calF$ that is generated by the open sets in $\Omega_\Lambda := \{1,\dots, q\}^{\Lambda}$. We will use the shorthand notation $\Lambda \Subset \Z^d$ to signify that $\Lambda$ is a finite subset of $\Z^d$. In the following we will often denote for a given configuration $\omega \in \Omega$ by $\omega_\Lambda$ its projection to the volume $\Lambda \subset \Z^d$ and write $\omega_\Lambda \omega_\Delta$ for the configuration on in $\Lambda \cup \Delta$ composed of $\omega_\Lambda$ and $\omega_\Delta$ for disjoint $\Lambda, \Delta \subset \Z^d$. For the special case $\Lambda = \{x\}$ we will also write $x^c = \Z^d \setminus \{x\}$ and $\omega_x\omega_{x^c}$. 
The set of probability measures on $\Omega$ will be denoted by $\calM_1(\Omega)$ and the space of continuous functions by $C(\Omega)$. For a configuration $\eta \in \Omega$ we will denote by $\eta^{x,i}$ the configuration that is equal to $\eta$ everywhere except at the site $x$ where it is equal to $i$. Moreover, for $\Lambda \subset \Z^d$ we will denote the corresponding cylinder sets by 
$
    [\eta_\Lambda] = \{\omega : \ \omega_\Lambda \equiv \eta_\Lambda \}. 
$
Whenever we are taking the probability of such a cylinder event with respect to some measure $\nu \in \calM_1(\Omega)$, we will omit the square brackets and simply write $\nu(\eta_\Lambda)$. 

\subsection{Interacting particle systems}

We will consider time-continuous Markovian dynamics on $\Omega$, namely interacting particle systems characterized by time-homogeneous generators $\mathscr{L}$ with domain $\text{dom}(\mathscr{L})$ and its associated Markov semigroup $(P_t)_{t \geq 0}$. 
For interacting particle systems we adopt the notation and exposition of the classical textbook \cite[Chapter 1]{liggett_interacting_2005}. 
In our setting, the generator $\mathscr{L}$ is given via a collection of  single-site transition rates $c_x(\eta, \xi_x)$, which are continuous in the starting configuration $\eta \in \Omega$. 
These rates can be interpreted as the infinitesimal rate at which the particle at site $x$ switches from the state $\eta_x$ to $\xi_x$, given that the rest of the system is currently in state $\eta_{x^c}$. 
The full dynamics of the interacting particle system is then given as the superposition of these local dynamics, i.e., 
\begin{align*}
    \mathscr{L}f(\eta) = \sum_{x \in \Z^d}\sum_{\xi_x}c_x(\eta, \xi_x)[f(\xi_x \eta_{x^c}) - f(\eta)].
\end{align*}
In \cite[Chapter 1]{liggett_interacting_2005} it is shown that the following two conditions are sufficient to guarantee the well-definedness. 
\begin{enumerate}[\bfseries (L1)]
    \item The rate at which the particle at a particular site changes its spin is uniformly bounded, i.e.,
    \begin{align*}
        \sup_{x \in \Z^d}\sum_{\xi_x}\norm{c_{x}(\cdot, \xi_{x})}_{\infty} < \infty 
    \end{align*}
    \item and the total influence of all other particles on a single particle is uniformly bounded, i.e.,
    \begin{align*}
        \sup_{x \in \Z^d}\sum_{y \neq x}\sum_{\xi_{x}}\delta_y\left(c_{x}(\cdot, \xi_{x})\right) < \infty, 
    \end{align*}
    where 
    \begin{align*}
        \delta_y(f) := \sup_{\eta, \xi\colon \eta_{y^c} = \xi_{y^c}}\abs{f(\eta)-f(\xi)}
    \end{align*}
    is the oscillation of a function $f:\Omega \to \R$ at the site $y \in \Z^d$. 
\end{enumerate}
Under these conditions, one can then show that the operator $\mathscr{L}$, defined as above, is the generator of a well-defined Markov process and that a core of $\mathscr{L}$ is given by the space of functions with finite total oscillation, i.e.
\begin{align*}
    D(\Omega) := \Big\{ f \in C(\Omega)\colon \sum_{x \in \Z^d} \delta_x(f) < \infty\Big\}.
\end{align*}

\noindent Let us emphasise briefly that we will not assume translation-invariance of the rates.

\subsection{Relative entropy loss}
For $\mu, \nu \in \calM_1(\Omega)$ and a finite volume $\Lambda \Subset \Z^d$ define the relative entropy of $\nu$ with respect to $\mu$ in $\Lambda$ via 
\begin{align*}
    h_\Lambda(\nu | \mu) :=
    \begin{cases}
    \sum_{\omega_\Lambda \in \Omega_\Lambda}\nu(\omega_{\Lambda})\log\frac{\nu(\omega_{\Lambda})}{\mu(\omega_\Lambda)}, \quad &\text{if } \nu_\Lambda \ll \mu_\Lambda, 
    \\\
    \infty, &\text{else,}
    \end{cases}
\end{align*}
where we use the convention that $0 \log 0 = 0$. Now recall that $(P_t)_{t \geq 0}$ denotes the Markov semigroup corresponding to the Markov generator $\mathscr{L}$. We write $\nu_t:=\nu P_t$ for the time-evolved measure $\nu\in\calM_1(\Omega)$.
The finite-volume relative entropy loss in $\Lambda \Subset \Z^d$ is defined by 
\begin{align*}
    g_\Lambda^{\mathscr{L}}(\nu|\mu) 
    := \frac{d}{dt}\Big \lvert_{t=0}h_{\Lambda}( \nu_t\lvert \mu). 
\end{align*}
%In the special case where $\Lambda$ is the centered cube $[-n,n]^d$ we will use the short-hand notation $g^n_{\mathscr{L}}(\cdot \lvert \mu)$. 
Usually, one then works with the density limits of the relative entropy and the relative entropy loss and shows that the latter can still be used as a Lyapunov function for the dynamics. However, the sub-additivity arguments that are used to show that the relative entropy loss density actually exists as a limit are only available for translation-invariant measures. We do not want to make any such assumptions and therefore we will have to instead work with the family of finite-volume relative entropy losses. Note that calling the finite-volume derivatives \textit{loss} is not entirely correct, since they are not necessarily non-positive. However one can show that the positive contributions are of boundary order and vanish in the density limit, see \cite[Lemma~21 and Lemma~23]{jahnel_dynamical_2023}. 

\subsection{Time-stationary measures, orbits, and the attractor}
If one is interested in the long-term behaviour of an interacting particle system, a natural object to study is the so-called \textit{attractor} of the measure-valued dynamics which is defined as
\begin{align*}
    \mathscr{A} = \set{\nu \in \calM_1(\Omega)\colon \exists \nu_0 \in \calM_1(\Omega) \text{ and } t_n \uparrow \infty \text{ such that } \lim_{n \to \infty}\nu_{t_n} = \nu}. 
\end{align*}
This is the set of all accumulation points of the measure-valued dynamics induced by $\mathscr{L}$. In the language of dynamical systems this is the $\omega$-limit set. This encodes (most of) the dynamically relevant information about the long-time behaviour of the system. 
In this article, we are particularly interested in two subsets of the attractor, namely the \textit{time-stationary measures} given by 
\begin{align*}
    \mathscr{S} := \set{\nu \in \calM_1(\Omega): \forall s \geq 0: \nu P_s = \nu},
\end{align*}
and the measures which lie on a \textit{stationary orbit}
\begin{align*}
    \mathscr{O} := \set{\nu \in \calM_1(\Omega): \exists T > 0: \nu P_T = \nu}. 
\end{align*}
The relation between these sets can be summarised as follows
\begin{align*}
    \mathscr{S} \subset \mathscr{O} \subset \mathscr{A}. 
\end{align*}
 In general, the first inclusion is strict as can be seen by considering the non-trivial examples constructed in \cite{jahnel_class_2014} and \cite{jahnel_time-periodic_2025} or the (from a probabilistic point of view) trivial example given in \cite[p.12]{liggett_interacting_2005}. 
Historically, most attention has been paid to investigating the set of time-stationary measures and their properties, but not much was known about the behaviour of interacting particle systems outside of this set.

\subsection{Main results}
Before we can state our main results, let us introduce some stronger conditions on the transition rates $(c_x(\cdot, \xi_x))_{x \in \Z^d, \xi_x \in \Omega_0}$ that will turn out to be crucial for our results. 

\noindent \textbf{Conditions for the rates.} 
\begin{enumerate}[\bfseries (R1)]
    \item  The rate at which the particle at a particular site changes its spin is uniformly bounded, i.e.,
    \begin{align*}
        \sup_{x \in \Z^d}\sum_{\xi_{x}}\norm{c_{x}(\cdot, \xi_{x})}_{\infty} < \infty. 
    \end{align*}
    \item For every $x \in \Z^d$ and $\xi_x \in \Omega_0$ the function 
    \begin{align*}
        \Omega \ni \eta \mapsto c_x(\eta, \xi_x) \in [0,\infty)
    \end{align*}
    is continuous.
    \item The transition rates are bounded away from zero, i.e.,
    \begin{align*}
        \inf_{x\in \Z^d,\, \eta \in \Omega,\, \xi_x \in \Omega_0 } c_x(\eta, \xi_x) =: \mathbf{c} >0. 
    \end{align*}
 %   \jk{Would be nice to weaken this assumption but it seems to be crucial in the proof of Lemma \ref{lemma:upper-bound-alpha}}
\item We have 
\begin{align*}
    \sum_{y \in \Z^d} \abs{y}\sup_{x \in \Z^d} \delta_{x+y}c_x(\cdot) < \infty,
\end{align*}
where $c_x(\eta) = \sum_{i \neq \eta_x}c_x(\eta,i)$ is the total rate at which the particle at site $x$ changes its state when the system is in configuration $\eta$. 
\end{enumerate}
The condition $\mathbf{(R4)}$ essentially ensures that the specification and the transition rates are \textit{short-range} and is satisfied if the dependence of $c_x$ on the spin at site $y \in \Z^d$ decays faster than $\abs{x-y}^{-2d}$.
Let us emphasise again that we do \textit{not} assume that the rates translation-invariant.  \medskip

Our main result is the following no-go result that essentially states that the presence of a product measure as a fixed point for the measure-valued dynamics makes it impossible to also possess periodic orbits. Note that admitting a product measure as a time-stationary measure does \textit{not} imply that the dynamics are in some sense trivial or non-interacting. The example of kinetically constrained models, see e.g.~\cite{hatarsky-toninelli}, shows that even interacting particle systems with very strong interactions can admit a product measure as a time-stationary measure.  

\begin{theorem}\label{theorem:main-result}
Let $d \in \{1,2\}$ and assume that $\mathscr{L}$ is the generator of an interacting particle system that satisfies assumptions $\mathbf{(R1)-(R4)}$ and that it admits a product measure $\mu$ as time-stationary measure.Then we have that 
\begin{align*}
    \mathscr{O} = \mathscr{S} = \{\mu\}. 
\end{align*}
\end{theorem}

While the first identity can be seen as the main result of this paper, the second identity also extends the results of \cite{ramirez_uniqueness_2002} from binary to general finite local state spaces and to unbounded range interactions up until power-law decay with exponent $\alpha > 2d$, see assumption $\mathbf{(R4)}$. 
 \medskip 

This structure theorem in particular implies the following no-go result that essentially states that the presence of a product measure as fixed point for the measure-valued dynamics makes it impossible to also possess periodic orbits. The precise formulation is as follows. 

\begin{corollary}\label{corollary:no-periodic}
Let $d \in \{1,2\}$ and assume that $\mathscr{L}$ satisfies the above assumptions and that $(P_t)_{t \geq 0}$ is the Markov semigroup generated by $\mathscr{L}$ with some product measure $\mu$ as time-stationary measure. Then, the measure-valued dynamics given by 
\begin{align*}
    [0,\infty) \times \calM_1(\Omega) \ni (t, \nu) \mapsto \nu P_t \in \calM_1(\Omega)
\end{align*}
does not contain non-trivial time-periodic orbits, i.e., there is no probability measure $\nu \in \calM_1(\Omega)$ such that $(\nu P_t)_{t\geq 0}$ is non-constant and such that there exists a $T>0$ with $\nu_t = \nu_{t +T}$. 
\end{corollary}

The proof strategy for Theorem \ref{theorem:main-result} and comments on possible extensions to more general time-stationary measures $\mu$ is discussed in Section \ref{section:proof-strategy}. 

\section{Outlook}\label{section:outlook}
Before diving into proof of the main result, let us briefly comment on some open problems related to time-periodic behaviour in interacting particle systems and possible strategies to establish (parts of) the conjecture mentioned in the introduction. 

\subsection{Towards a no-go theorem in $d=1,2$}
At least on an intuitive level one could be led to believe that the stationary measures of interacting particle system with sufficiently nice rates, say finite-range and strictly positive, should also be somewhat regular. From this perspective, it is not much of a stretch to conjecture that generically there should be at least one stationary measure $\mu$ which is quasilocal, i.e., it admits a version of its local conditional distributions $\mu_\Lambda[\cdot \lvert \cdot]$ which is \textit{continuous with respect to the boundary condition}. In the present setting of finite local state spaces this in particular implies that $\mu$ is a Gibbs measure for some absolutely summable Hamiltonian $\mathcal{H}$. The set of translation-invariant Gibbs measures for a Hamiltonian $\mathcal{H}$ will be denoted by $\mathscr{G}_\theta(\mathcal{H})$. 

Assuming that this quasilocality is generically given, if one now restricts attention to the translation-invariant case, then the following proposition is a consequence of a more general result in \cite{jahnel_dynamical_2023}. 

\begin{proposition}
    If a translation-invariant interacting particle system with generator $\mathscr{L}$ that satisfies $\mathbf{(L1)-(L2)}$ and $\mathbf{(R3)}$ admits a time-stationary Gibbs measure $\mu \in \mathscr{G}_\theta(\mathcal{H})$, where $\mathcal{H}$ is absolutely summable, then any time-periodic orbit that consists of translation-invariant measures is necessarily contained in $\mathscr{G}_\theta(\mathcal{H})$. 
\end{proposition}

By using an argument from \cite{kunsch_non_1984} based on properties of extremal measures we obtain the following connection between time-periodic behaviour and a property of the set of extremal Gibbs measures.  

\begin{proposition}
    If a translation-invariant interacting particle system satisfying the conditions $\mathbf{(L1)-(L2)}$ and $\mathbf{(R3)}$ admits a time-stationary measure $\mu \in \mathscr{G}_\theta(\mathcal{H})$, where $\mathcal{H}$ is absolutely summable, and if  $\emph{ex}\mathscr{G}_\theta(\mathcal{H})$ is finite, then every $\nu \in \mathscr{G}_\theta(\mathcal{H})$ is time-stationary for the dynamics. In particular, there can be no non-trivial time-periodic orbits. 
\end{proposition}

\begin{proof}
    A translation-invariant Gibbs measure $\nu$ is an extremal point if and only if it is ergodic with respect to spatial shifts. Moreover, by \cite[Theorem I.4.15]{liggett_interacting_2005} if $\nu_0$ is ergodic, then so is $\nu_t$ for every $t \geq 0$. So if we choose an extremal Gibbs measure $\nu$ as initial condition, then by monotonicity of the free energy density, see \cite[Theorem 6]{jahnel_dynamical_2023}, we will remain in the set of extremal Gibbs measures for all times $t \geq 0$. But since $(\nu_t)_{t \geq 0}$ is continuous as a function of time and $\text{ex}\mathscr{G}_\theta(\mathcal{H})$ is finite and in particular totally disconnected, this can only be the case if $\nu_t = \nu$ for all times $t\geq 0$.  By using the extremal decomposition for Gibbs measures, see e.g. \cite[Section 6.8.4]{friedli_statistical_2017}, the time-stationarity of all extremal translation-invariant Gibbs measures implies that \textit{any} translation-invariant Gibbs measure is time-stationary for the process. In particular, there can be no non-trivial time-periodic orbits. 
\end{proof}

There is a general conjecture, see e.g. \cite{van_enter_aperiodicity_2014}, that in two dimensions finite-range models should always possess a finite number of extremal Gibbs states, all of which are translation-invariant. In particular, this would imply that all Gibbs states are translation-invariant. This would then also imply that there can be no non-trivial time-periodic behaviour in two-dimensional interacting particle systems that admit a time-stationary Gibbs measure.
In \cite{dobrushin-shlosman}, the authors make use of Pirogov--Sinai theory to show that at least for very low temperatures, this general conjecture holds. But the intermediate temperature regime is still open. 
For the Potts model, the main result of \cite{coquille_gibbs_2014} in particular implies that for any temperature $\beta\geq 0$, there are only finitely many extremal Gibbs measures, but the technique does not easily extend to more general situations. 

In higher dimensions, this does not need to be the case as e.g. the example constructed in \cite[Section 4]{slawny_family_1974} shows. Indeed, despite only having a finite local state space and finite-range interactions in $d\geq 3$, there is not just an infinite but even an \textit{uncountable} number of translation-invariant extremal Gibbs measures at sufficiently low temperatures. 

To sum up the preceding discussion, at least in the translation-invariant case, one possible strategy to rule out stable time-periodic orbits in one and two dimensions would be to show that (a) any sufficiently nice interacting particle system admits a time-stationary Gibbs measure (with a sufficiently regular Hamiltonian) and (b) for sufficiently nice Hamiltonians,  the set of translation-invariant extremal Gibbs measures is always finite in two dimensions. 

\subsection{Towards a minimal example}
While the examples constructed in \cite{jahnel_class_2014} and \cite{jahnel_time-periodic_2025} show that $\mathbf{(TTSB)}$ is indeed possible in finite dimensions, even in non-degenerate interacting particle systems, their definition is quite implicit and not particularly well suited for further analysis. Therefore, it is desirable to find a sufficiently simple toy model, that (a) exhibits time-periodic behaviour, (b) is amenable to mathematical analysis, and (c) is non-trivial. 

Numerical simulations suggest that periodic behavior is not a rare phenomenon. Several interacting particle systems with explicit dynamics are known to exhibit periodic behavior in simulations, at least in three and higher dimensions. 
Three notable examples are the \textit{nonreciprocal Ising model}, see e.g.\cite{collet_rhythmic_2016, avni_nonreciprocal_2025,guislain_collective_2024}, the \textit{cycle conform model}, see \cite[Section 1.9]{swart_course_2026}, and \textit{(self-)driven clock models}, see \cite{maes_rotating_2011,wu_synchronization_2026}. One feature all of these examples share is that they are all non-reversible. 
While the numerical evidence for the existence of time-periodic behaviour in all these models is quite substantial, at the moment only their mean-field counterparts can be rigorously shown to admit time-periodic solutions.

\section{Proof strategy}\label{section:proof-strategy}
The proof strategy for Theorem \ref{theorem:main-result} is inspired by \cite{jahnel_long-time_2025} and essentially consists of two main steps that we will now explain briefly before we start with the actual mathematics. Let us already note that apart from the very last step in the proof of Proposition \ref{proposition:time-averaged-non-reversible-holley-stroock} all the technical results derived in the forthcoming sections hold in any dimension $d\in \N$ and can therefore be used for future investigations in arbitrary dimensions. Moreover, only one building block, namely Lemma \ref{lemma:monotonicity-alpha}, seems to require the product structure of $\mu$. 

\subsection{Finite-volume relative entropy loss and stationary orbits}
The first technical result is the following \textit{time-averaged} version of the results in \cite{ramirez_uniqueness_2002} which also extends the classical result beyond finite-range interactions and from binary to general finite local state spaces. 

The main difference to \cite[Proposition 4.1]{jahnel_long-time_2025} is that we replace the assumption of reversibility, which provides a local relation between the rates and the reversible measure, by the assumption that the process leaves some product measure invariant. 

\begin{proposition}[Time-averaged entropy loss principle]\label{proposition:time-averaged-non-reversible-holley-stroock}
    Assume that $d\in \{1,2\}$, that $\mathscr{L}$ satisfies $\mathbf{(R1)-(R4)}$ and admits a stationary measure $\mu$ which is a product measure. If $\nu \in \calM_1(\Omega)$ is such that 
    \begin{enumerate}[\bfseries (M1)]
        \item for all $\eta \in \Omega, \Lambda \Subset \Z^d$, and $s\geq 0$ it holds that $\nu P_s(\eta_\Lambda)>0$, and 
        \item for all $\Lambda \Subset \Z^d$ we have $\int_0^T g_\Lambda^\mathscr{L}(\nu P_s \lvert \mu) ds = 0$,
    \end{enumerate}
    then we have $\nu = \mu$. 
\end{proposition}

The proof of this can be found in Section \ref{section:proof-relative-entropy-loss-principle} and morally follows a similar strategy as in \cite{jahnel_long-time_2025} but the details are quite different, since we do not have reversibility at our hands. Therefore, we make use of  a different rewriting of the finite-volume relative entropy loss which first appeared in  \cite{ramirez_uniqueness_2002}. 

Since condition $\mathbf{(M2)}$ is directly implied by periodicity and the fundamental theorem of calculus, our main result Theorem \ref{theorem:main-result} will follow by showing that $\mathbf{(M1)}$ is satisfied along time-periodic orbits. But this was already established in \cite{jahnel_long-time_2025}. Let us nevertheless recall the main ideas. 

\subsection{The positive-mass property}

In the simple setting of an irreducible continuous-time Markov chain on finite state space $X$, it is easy to show that there exist constants $\rho, \tau >0$ such that for any initial distribution $\nu$ and all states $x \in X$ and times $t\geq \tau >0$, the probability of being in state $x$ at time $t$ is at least $\rho$, i.e, we have $\nu_t(x)\geq \rho >0$. In the setting of infinite-volume interacting particle systems that are irreducible in a suitable way, an analogous property should be true, but here it is not as straightforward to see. 

The key idea to make this intuition precise for infinite volume systems is to compare the dynamics to an interacting particle system in which all of the sites inside of a finite volume $\Lambda$ behave independently and flip with the minimal transition rate. A Girsanov-type formula then allows one to compare this finite-volume perturbation with the original system and yields the following result.

\begin{proposition}[Positive-mass property] Assume that the rates of an interacting particle systems satisfy assumptions $\mathbf{(R1)-(R4)}$. Then, for all $\tau>0$ and $\Lambda \Subset \Z^d$, there exists a constant $C(\tau,\Lambda)>0$ such that, for any starting measure $\nu$ and any time $t\in[\tau,\infty)$, we have 
\begin{align*}
    \forall \eta \in \Omega \quad \nu_t(\eta_\Lambda) \geq C(\tau, \Lambda). 
\end{align*}
In particular, for all subsequential limits $\nu^* = \lim_{n \to \infty}\nu_{t_n}$ with $t_n \uparrow \infty$, we have 
\begin{align*}
    \forall \eta \in \Omega \forall \Lambda \Subset \Z^d \quad \nu^*(\eta_\Lambda) \geq C(\tau,\Lambda) >0. 
\end{align*}
\end{proposition}

The proof of this can be found in \cite{jahnel_long-time_2025}. Note that this in particular implies that the condition $\mathbf{(M1)}$ holds along time-periodic orbits and that the positive-mass property can be interpreted as a somewhat quantitative version of the noisy nature of the dynamics. Even if we start our process with a point mass $\delta_\omega$ at some configuration $\omega \in \Omega$ as initial condition, the distribution of the process at any positive time $t>0$ will already put positive mass on \textit{any} cylinder set $[\eta_\Lambda]$. 

\section{Proof of the time-averaged relative entropy loss principle}\label{section:proof-relative-entropy-loss-principle}

Unless stated otherwise, we will from now on always assume that the transition rates of our dynamics satisfy $\mathbf{(R1)-(R4)}$. 

\subsection{Rewriting the finite-volume relative entropy loss}
We start our proof by obtaining a more convenient representation of the relative entropy loss in finite volumes $\Lambda \Subset \Z^d$. Recall that the relative entropy loss in the finite volume $\Lambda$ is defined by
\begin{align*}
    g_\Lambda^\mathscr{L}(\nu \lvert \mu) = \frac{d}{dt}\lvert_{t=0}h_\Lambda(\nu \lvert \mu), \quad \nu \in \calM_1(\Omega).
\end{align*}
The first step is to rewrite this in a suitable way that will allow us to separate the (negative) contribution from the bulk from the (possibly positive) finite volume error terms. The expression we derive here is different from the one in \cite[Lemma 5.3]{jahnel_long-time_2025} and more similar to the one in \cite[Lemma 23]{jahnel_dynamical_2023}. However, in the latter reference, we were just interested in the rough asymptotics of the error term and proved that it is of boundary order. This is not sufficient for our goal in this article and we therefore have to keep track of the error term and will later make use of its explicit form. 

\begin{lemma}\label{lemma:finite-volume-relative-entropy-loss-non-reversible}
    For $\Lambda\Subset \Z^d$ and $\nu \in \calM_1(\Omega)$ with $\nu(\eta_\Lambda)>0$ for all $\eta_\Lambda \in \Omega_\Lambda$ we have 
    \begin{align*}
        g_\Lambda^\mathscr{L}(\nu\lvert \mu) 
    =
    &-\sum_{\eta_\Lambda}\sum_{x \in \Lambda}\sum_{j \neq \eta_x}F\left(\frac{\nu(\eta_\Lambda)}{\nu(\eta_\Lambda^{x,j})}\frac{\mu(\eta_\Lambda^{x,j})}{\mu(\eta_\Lambda)}\right)\frac{\mu(\eta_\Lambda)}{\mu(\eta_{\Lambda}^{x,j})}\frac{\nu(\eta^{x,j}_\Lambda)}{\nu(\eta_\Lambda)}\int_{\eta_\Lambda}c_x(\omega,j)\nu(d\omega)
    \\\
    -
    &\sum_{\eta_\Lambda}\sum_{x \in \Lambda}\sum_{j \neq \eta_x}\left[\int_{\eta_\Lambda}c_x(\omega,j)\nu(d\omega) - \frac{\mu(\eta_\Lambda)}{\mu(\eta_\Lambda^{x,j})}\frac{\nu(\eta^{x,j}_\Lambda)}{\nu(\eta_\Lambda)}\int_{\eta_\Lambda}c_x(\omega,j)\nu(d\omega)\right] 
\end{align*}
where we use the notation $F(x) = x \log x - x + 1$ for $x\geq 0$ with the convention that $0\log 0 = 0$. 
\end{lemma}

\begin{proof}
    A direct calculation using the definition of the generator yields 
    \begin{align}\label{eqn:starting-point}
        \nonumber 
        g_\Lambda^\mathscr{L}(\nu|\mu) 
        &=  \sum_{\eta_\Lambda}\nu\left(\mathscr{L}\mathbf{1}_{\eta_\Lambda}\right)\log\left(\frac{\nu(\eta_\Lambda)}{\mu(\eta_\Lambda)}\right)
        \\\
        &=
        \sum_{\eta_\Lambda}\sum_{x\in \Lambda}\sum_{j}\int c_x(\omega,j)\left[\mathbf{1}_{\eta_\Lambda}(\omega^{x,j})-\mathbf{1}_{\eta_\Lambda}(\omega)\right]\nu(d\omega)\log\left(\frac{\nu(\eta_\Lambda)}{\mu(\eta_\Lambda)}\right)
        \\\
        &=
        \sum_{\eta_{\Lambda}}\sum_{x \in \Lambda}\sum_{i\neq \eta_x}
        \left[\int_{[\eta_{\Lambda}^{x,i}]}c_x(\omega, \eta_x)\nu(d\omega) - \int_{[\eta_{\Lambda}]}c_x(\omega,i)\nu(d\omega)\right]\log\left(\frac{\nu(\eta_{\Lambda})}{\mu(\eta_{\Lambda})}\right),
        \nonumber
    \end{align}
    where we used that 
    \begin{align*}
        \text{if $j\neq \eta_x$, then } &\int c_x(\omega,j)\left[\mathbf{1}_{\eta_\Lambda}(\omega^{x,j})-\mathbf{1}_{\eta_\Lambda}(\omega)\right]\nu(d\omega) = -\int_{[\eta_\Lambda]}c_x(\omega,j)\nu(d\omega), 
        \\\
        \text{ and if $j=\eta_x$, then } &\int c_x(\omega,j)\left[\mathbf{1}_{\eta_\Lambda}(\omega^{x,j})-\mathbf{1}_{\eta_\Lambda}(\omega)\right]\nu(d\omega) = \sum_{i \neq \eta_x}\int_{[\eta_\Lambda^{x,i}]}c_x(\omega,\eta_x)\nu(d\omega). 
    \end{align*}
    By rearranging the summation slightly, the identity \eqref{eqn:starting-point} yields 
    \begin{align*}
        g_\Lambda^\mathscr{L}(\nu|\mu) 
        &=
        \sum_{\eta_{\Lambda}}\sum_{x \in \Lambda}\sum_{i\neq \eta_x}
         \log\left(\frac{\mu(\eta_\Lambda)\nu(\eta_{\Lambda}^{x,i})}{\nu(\eta_\Lambda)\mu(\eta_\Lambda^{x,i})}\right)\int_{[\eta_\Lambda]}c_x(\omega,i)\nu(d\omega).
    \end{align*}
    By adding and subtracting suitable terms this now yields the claimed expression. 
\end{proof}

\subsection{The time-averaged zero-loss estimate}
The previously derived representation of the relative entropy loss in finite volumes $\Lambda \Subset \Z^d$ leads to the following inequality in case the time-averaged relative entropy loss vanishes.

\begin{lemma}\label{lemma:time-averaged-zero-loss-ineq}
Let $\Lambda \Subset \Z^d$ and $\nu \in \calM_1(\Omega)$ be such that there exists $T>0$ with 
\begin{enumerate}[i.]
    \item $\nu P_s(\eta_{\Lambda})>0$ for all $\eta \in \Omega_\Lambda$ and $s\in [0,T]$,
    \item and $\int_0^T g_\Lambda^\mathscr{L}(\nu P_s \lvert \mu)ds = 0$.
\end{enumerate}  
Then for the constant $\mathbf{c}>0$  from $\mathbf{(R3)}$, that does in particular not depend on $\Lambda$ and $\nu$, we have 
\begin{align*}
    &\frac{\mathbf{c}}{2}\int_0^T\sum_{\eta_\Lambda}\sum_{x \in \Lambda}\sum_{j \neq \eta_x}\left(\sqrt{\frac{\nu_s(\eta^{x,j}_\Lambda)}{\mu(\eta_{\Lambda}^{x,j})}}-\sqrt{\frac{\nu_s(\eta_\Lambda)}{\mu(\eta_\Lambda)}}\right)^2 \mu(\eta_\Lambda)ds
    \\\
    \leq 
    & 2 \int_0^T
    \sum_{\eta_\Lambda}\sum_{x \in \Lambda}\sum_{j\neq \eta_x} \gamma_\Lambda(x) \abs{\frac{\nu_s(\eta_\Lambda^{x,j})}{\mu(\eta_\Lambda^{x,j})}-\frac{\nu_s(\eta_\Lambda)}{\mu(\eta_\Lambda)}}\mu(\eta_\Lambda)ds,  
\end{align*}
where we use the notation 
\begin{align*}
    \gamma_\Lambda(x) := \sum_{y \notin \Lambda}\sum_{j=1}^q \delta_y \left( c_x(\cdot)\right).  
\end{align*}
In the case where $\Lambda = \Lambda_n = [-n,n]^d \cap \Z^d$, we will use the shorthand $\gamma_n$ instead of $\gamma_{\Lambda_n}$. 
\end{lemma}

For the proof of Lemma \ref{lemma:time-averaged-zero-loss-ineq} we make use of the following technical helper which is reminiscent of Lebesgue's differentiation theorem and already appeared in \cite{jahnel_long-time_2025}. It will also come in handy later, so we record it for future reference and give a quick proof.  
\begin{lemma}[Quantitative differentiation lemma]\label{lemma:quantitative-differentiation}
    Let $\nu$ be a probability measure such that $\nu(\eta_{\Lambda})>0$ for all $\eta \in \Omega$ and $\Lambda \Subset \Z^d$. Then, for any function $f:\Omega \to \R$ with the property 
    \begin{align*}
        \sum_{x \in \Z^d}\delta_x f < \infty,
    \end{align*}
    the following uniform error estimate holds for all $\eta \in \Omega$
    \begin{align*}
        \abs{\frac{1}{\nu(\eta_\Lambda)}\int_{[\eta_\Lambda]}f(\omega)\nu(d\omega) - f(\eta)}
        \leq 
        \sum_{x \notin \Lambda}\delta_x f. 
    \end{align*}
\end{lemma}

\begin{proof}
    Fix $\eta \in \Omega$ and $\Lambda \Subset \Z^d$. Then we can fix an enumeration of the vertices in $\Lambda^c$ and write 
    $$[n] = \Lambda \cup \{x_1, \dots, x_n\}.$$ 
    By a telescope sum we see
    \begin{align*}
        f(\eta) - f(\omega)
        = \sum_{n=1}^\infty\left(f(\eta_{[n]}\omega_{[n]^c})-f(\eta_{[n-1]}\omega_{[n-1]^c})\right)
        \leq 
        \sum_{x \notin \Lambda}\delta_x f. 
    \end{align*}
    The claim now follows via integration. 
\end{proof}

\begin{proof}[Proof of Lemma \ref{lemma:time-averaged-zero-loss-ineq}]
    From the representation derived in Lemma \ref{lemma:finite-volume-relative-entropy-loss-non-reversible} and the assumption that $\int_0^T g^n_\mathscr{L}(\nu P_s \lvert \mu)ds = 0$, we directly obtain the identity 
    \begin{align}\label{eq:key-identity}
        &\int_0^T \sum_{\eta_\Lambda}\sum_{x \in \Lambda}\sum_{j \neq \eta_x}F\left(\frac{\nu_s(\eta_\Lambda)}{\nu_s(\eta_\Lambda^{x,j})}\frac{\mu(\eta_\Lambda^{x,j})}{\mu(\eta_\Lambda)}\right)\frac{\mu(\eta_\Lambda)}{\mu(\eta_{\Lambda}^{x,j})}\frac{\nu_s(\eta^{x,j}_\Lambda)}{\nu_s(\eta_\Lambda)}\int_{\eta_\Lambda}c_x(\omega,j)\nu_s(d\omega) \ ds
        \\\
        =
    -&\int_0^T\sum_{\eta_\Lambda}\sum_{x \in \Lambda}\sum_{j \neq \eta_x}\left[\int_{\eta_\Lambda}c_x(\omega,j)\nu_s(d\omega) - \frac{\mu(\eta_\Lambda)}{\mu(\eta_\Lambda^{x,j})}\frac{\nu_s(\eta^{x,j}_\Lambda)}{\nu_s(\eta_\Lambda)}\int_{\eta_\Lambda}c_x(\omega,j)\nu_s(d\omega)\right] \ ds. \nonumber
    \end{align}
    Since $F(x) \geq \tfrac{1}{2}(1-\sqrt{x})^2$ for all $x>0$, the left-hand side of this identity can be bounded from below by
    \begin{align*}
        &\int_0^T \sum_{\eta_\Lambda}\sum_{x \in \Lambda}\sum_{j \neq \eta_x}F\left(\frac{\nu_s(\eta_\Lambda)}{\nu_s(\eta_\Lambda^{x,j})}\frac{\mu(\eta_\Lambda^{x,j})}{\mu(\eta_\Lambda)}\right)\frac{\mu(\eta_\Lambda)}{\mu(\eta_{\Lambda}^{x,j})}\frac{\nu_s(\eta^{x,j}_\Lambda)}{\nu(\eta_\Lambda)}\int_{\eta_\Lambda}c_x(\omega,j)\nu_s(d\omega) \ ds
        \\\
        \geq
        &\frac{1}{2}\int_0^T\sum_{\eta_\Lambda}\sum_{x \in \Lambda}\sum_{j \neq \eta_x}\left(1-\sqrt{\frac{\nu_s(\eta_\Lambda)}{\nu_s(\eta_\Lambda^{x,j})}\frac{\mu(\eta_\Lambda^{x,j})}{\mu(\eta_\Lambda)}}\right)^2 \frac{\mu(\eta_\Lambda)}{\mu(\eta_{\Lambda}^{x,j})}\frac{\nu_s(\eta^{x,j}_\Lambda)}{\nu_s(\eta_\Lambda)}\int_{\eta_\Lambda}c_x(\omega,j)\nu_s(d\omega)\ ds 
        \\\
        \geq 
        &\frac{\mathbf{c}}{2}\int_0^T\sum_{\eta_\Lambda}\sum_{x \in \Lambda}\sum_{j \neq \eta_x}\left(\sqrt{\frac{\nu_s(\eta_\Lambda^{x,j})}{\mu(\eta_\Lambda^{x,j})}}-\sqrt{\frac{\nu_s(\eta_\Lambda)}{\mu(\eta_\Lambda)}}\right)\mu(\eta_\Lambda)\ ds
    \end{align*}

    To deal with the right-hand side of the identity \eqref{eq:key-identity} we first note that by stationarity and non-nullness of $\mu$ we have for any $\rho \in \calM_1(\Omega)$
    \begin{align}\label{eq:invariance-equation-mu}
    0 
    =
    \int_\Omega \mathscr{L} \left(\frac{d\rho}{d\mu}\Big\lvert_{\mathcal{F}_\Lambda}\right)(\omega) \mu(d\omega)
    &=
    \sum_{\eta_\Lambda}\sum_{x \in \Lambda}\sum_{j\neq \eta_x}\int_{[\eta_\Lambda]}c_x(\omega,j)\left(\frac{\rho(\eta_\Lambda^{x,j})}{\mu(\eta^{x,j}_\Lambda)}-\frac{\rho(\eta_\Lambda)}{\mu(\eta_\Lambda)}\right)\mu(d\omega)
    \\\
    &=
    -\sum_{\eta_\Lambda}\sum_{x \in \Lambda}\sum_{j\neq \eta_x}
    \mu(\eta_{\Lambda})c^{\Lambda,\mu}_x(\eta_\Lambda,j)\left[\frac{\rho(\eta_\Lambda^{x,j})}{\nu(\eta_\Lambda^{x,j})} - \frac{\mu(\eta_\Lambda)}{\mu(\eta_\Lambda^{x,j})}\right]
    \end{align}
    where we use the notation 
    \begin{align*}
        c^{\Lambda,\rho}_x(\eta_\Lambda,j) := \frac{1}{\rho(\eta_\Lambda)}\int_{[\eta_\Lambda]}c_x(\omega,j)\rho(d\omega), \quad \eta_\Lambda \in \Omega_\Lambda, 
    \end{align*}
    for a probability measure $\rho \in \calM_1(\Omega)$, $\Lambda \Subset \Z^d$, and $j \in \Omega_0$. By using the same notation, the right-hand side of \eqref{eq:key-identity} can be rewritten as
    \begin{align*}
        \int_0^T \sum_{\eta_\Lambda}\sum_{x \in \Lambda}\sum_{j \neq \eta_x}
        \mu(\eta_\Lambda) c^{\Lambda,\nu_s}_x(\eta_\Lambda,j)\left(\frac{\nu_s(\eta_\Lambda)}{\mu(\eta_\Lambda)}-\frac{\nu_s(\eta_\Lambda^{x,j})}{\mu(\eta_\Lambda^{x,j})}\right). 
    \end{align*}
    So we can use \eqref{eq:invariance-equation-mu} for every $\nu_s$, where $s \in [0,T]$, and combine it with \eqref{eq:key-identity} to obtain 
    \begin{align}\label{ineq:almost-done}
    &\int_0^T \sum_{\eta_\Lambda}\sum_{x \in \Lambda}\sum_{j \neq \eta_x}F\left(\frac{\nu_s(\eta_\Lambda)}{\nu_s(\eta_\Lambda^{x,j})}\frac{\mu(\eta_\Lambda^{x,j})}{\mu(\eta_\Lambda)}\right)\frac{\mu(\eta_\Lambda)}{\mu(\eta_{\Lambda}^{x,j})}\frac{\nu_s(\eta^{x,j}_\Lambda)}{\nu_s(\eta_\Lambda)}\int_{\eta_\Lambda}c_x(\omega,j)\nu_s(d\omega) \ ds
        \\\
        =
        &\int_0^T \sum_{\eta_\Lambda}\sum_{x \in \Lambda}\sum_{j \neq \eta_x}
        \mu(\eta_\Lambda) \left(c^{\Lambda,\mu}_x(\eta_\Lambda,j)-c^{\Lambda,\nu_s}_x(\eta_\Lambda,j)\right) \left(\frac{\nu_s(\eta_\Lambda^{x,j})}{\mu(\eta_\Lambda^{x,j})}-\frac{\nu_s(\eta_\Lambda)}{\mu(\eta_\Lambda)}\right).
    \end{align}
    By Lemma \ref{lemma:quantitative-differentiation}, we get the following \textit{uniform} estimate for any $\rho \in \calM_1(\Omega)$ and $\eta_\Lambda \in \Omega_\Lambda$
    \begin{align*}
        \abs{c^{\Lambda,\mu}_x(\eta_\Lambda,j)-c^{\Lambda,\rho}_x(\eta_\Lambda,j)} \leq 2 \sum_{y\notin \Lambda}\delta_y\left(c_x(\cdot, j)\right), \quad x \in \Lambda, j\in \Omega_0. 
    \end{align*}
    Now combining this with a simple application of the triangle inequality to \eqref{ineq:almost-done} yields the claimed estimate. 
\end{proof}

Motivated by the inequality derived in Lemma \ref{lemma:time-averaged-zero-loss-ineq}, let us introduce the following notation for $\rho \in \calM_1(\Omega)$, $\Lambda \Subset \Z^d$, and $x\in \Lambda$: 
\begin{align*}
    \alpha_\Lambda(x,\rho) &:= \sum_{\eta_\Lambda}\sum_{j\neq \eta_x}\left(\sqrt{\frac{\rho(\eta^{x,j}_\Lambda)}{\mu(\eta_{\Lambda}^{x,j})}}-\sqrt{\frac{\rho(\eta_\Lambda)}{\mu(\eta_\Lambda)}}\right)^2 \mu(\eta_\Lambda), 
    \\\
    \beta_\Lambda(x,\rho) &:= \sum_{\eta_\Lambda}\sum_{j\neq \eta_x}\abs{\frac{\rho(\eta_\Lambda^{x,j})}{\mu(\eta_\Lambda^{x,j})}-\frac{\rho(\eta_\Lambda)}{\mu(\eta_\Lambda)}}\mu(\eta_\Lambda).
\end{align*}
In the case where $\Lambda = \Lambda_n$, we will use the shorthand $\alpha_n$ and $\beta_n$ instead of $\alpha_{\Lambda_n}$ and $\beta_{\Lambda_n}$. 
With this notation at hand, we can express the result from Lemma \ref{lemma:time-averaged-zero-loss-ineq} as
\begin{align}\label{ineq:stationarity-derived-inequality}
    \frac{\mathbf{c}}{2}\sum_{x \in \Lambda}\int_0^T \alpha_\Lambda(x,\nu_s) ds 
    \leq 
    2 \sum_{x \in \Lambda} \gamma_\Lambda(x) \int_0^T \beta_\Lambda(x,\nu_s)ds. 
\end{align}
To get some intuition on the above inequality, note that in the case of nearest-neighbour interactions, the coefficients $\gamma_\Lambda(x)$ vanish if $x \notin \partial \Lambda$, so the sum on the right-hand side of \eqref{ineq:stationarity-derived-inequality} can be seen as a boundary contribution, whereas the sum on the left is the bulk contribution. 
Moreover, the $\alpha_\Lambda(x,\cdot)$ are clearly non-negative, continuous with respect to the weak topology on $\calM_1(\Omega)$ and one has 
\begin{align*}
    \rho = \mu \quad \iff \quad \forall \Lambda \Subset \Z^d \ \forall x \in \Lambda: \ \alpha_\Lambda(x,\rho) = 0. 
\end{align*}
It will therefore be our goal to show that these coefficients actually do vanish along time-periodic orbits.

\subsection{Properties of the $\gamma$-coefficients}
In the general case, where the interaction range is not bounded, the intuition as stated above can be made precise as follows. 

\begin{lemma}\label{lemma:properties-rho}
    Under assumption $\mathbf{(R4)}$ it holds that 
    \begin{align*}
        \mathbf{C}_1 := \sup_{x \in \Z^d} \sum_{n=1}^\infty \gamma_n(x) < \infty, 
        \quad \text{and} \quad 
        \mathbf{C_2} := \sup_{n \in \N}\frac{1}{n^{d-1}}\sum_{x \in \Lambda_n}\gamma_n(x) < \infty. 
    \end{align*}
\end{lemma}

\begin{proof}
    \textit{Ad $\mathbf{C}_1$}: For fixed $x\in\Z^d$ we have 
\begin{align*}
    \sum_{n=1}^\infty \gamma_n(x)
    &=
    \sum_{n=1}^\infty \sum_{y \notin \Lambda_n} \delta_y c_x(\cdot)
    \\\
    &=
    \sum_{y \in \Z^d}\delta_y c_x(\cdot)\abs{\{n \in \N: \ x \in \Lambda_n, y \notin \Lambda_n\}}
    \\\
    &\leq 
    \sum_{y \in \Z^d}\delta_y c_x(\cdot) \abs{x-y}.
\end{align*}
Now assumption $\mathbf{(R4)}$ yields a uniform in $x$ upper bound on this quantity. 

\medskip 
\noindent 
\textit{Ad $\mathbf{C}_2$}: Here we have for fixed $n \in  \N$
\begin{align*}
    \sum_{x \in  \Lambda_n}\sum_{y \notin \Lambda_n} \delta_y c_x(\cdot)
    &\leq
    \sum_{v \in \Z^d}\sum_{x \in \Lambda_n: \ x+v \notin \Lambda_n}\delta_{x+v}c_x(\cdot)
    \\\
    &\leq 
    d(2n+1)^{d-1} \sum_{v \in \Z^d}\abs{v}\sup_{x \in \Z^d}\delta_{x+v}c_x(\cdot)).
\end{align*}
This can be bounded from above, independent of $n$, by assumption $\mathbf{(R4)}$. 
\end{proof}
\subsection{Pointwise bound of the $\beta$-coefficients in terms of the $\alpha$-coefficients}

\begin{lemma}\label{lemma:upper-bound-beta}
    There exists a constant $C>0$ such that for all $\rho \in \calM_1(\Omega)$, $\Lambda \Subset \Z^d$ and $x \in \Lambda$ it holds that 
    \begin{align*}
        \beta_\Lambda(x,\rho)^2 \leq C \cdot \alpha_\Lambda(x,\rho). 
    \end{align*}
\end{lemma}

\begin{proof}
    By definition of $\beta_{\Lambda}(x,\rho)$ and the Cauchy--Schwarz inequality we have 
    \begin{align*}
        \beta_\Lambda(x,\rho)^2
        &=
        \left(\sum_{\eta_\Lambda}\sum_{j \neq \eta_x}\mu(\eta_\Lambda)\left(\sqrt{\frac{\rho(\eta_\Lambda^{x,j})}{\mu(\eta_\Lambda^{x,j})}}-\sqrt{\frac{\rho(\eta_\Lambda)}{\mu(\eta_\Lambda)}}\right)\left(\sqrt{\frac{\rho(\eta_\Lambda^{x,j})}{\mu(\eta_\Lambda^{x,j})}}+\sqrt{\frac{\rho(\eta_\Lambda)}{\mu(\eta_\Lambda)}}\right)\right)^2
        \\\
        &\leq 
        \sum_{\eta_\Lambda}\sum_{j \neq \eta_x}\mu(\eta_\Lambda)\left(\sqrt{\frac{\rho(\eta_\Lambda^{x,j})}{\mu(\eta_\Lambda^{x,j})}}-\sqrt{\frac{\rho(\eta_\Lambda)}{\mu(\eta_\Lambda)}}\right)^2 
        \\\
        &\ + \sum_{\eta_\Lambda}\sum_{j \neq \eta_x}\mu(\eta_\Lambda) \left(\sqrt{\frac{\rho(\eta_\Lambda^{x,j})}{\mu(\eta_\Lambda^{x,j})}}+\sqrt{\frac{\rho(\eta_\Lambda)}{\mu(\eta_\Lambda)}}\right)^2
        \\\
        &= 
        \alpha_{\Lambda}(x,\rho) \sum_{\eta_\Lambda}\sum_{j \neq \eta_x}\mu(\eta_\Lambda) \left(\sqrt{\frac{\rho(\eta_\Lambda^{x,j})}{\mu(\eta_\Lambda^{x,j})}}+\sqrt{\frac{\rho(\eta_\Lambda)}{\mu(\eta_\Lambda)}}\right)^2
    \end{align*}
    So it suffices to bound the second factor on the right-hand side by a constant which does not depend on $\rho, \Lambda$ and $x$. To do this, first note that for all $a,b \geq 0$ one has $(a+b)^2 \leq 2(a^2+b^2)$, hence 
    \begin{align*}
        \sum_{\eta_\Lambda}\sum_{j \neq \eta_x}\mu(\eta_\Lambda) \left(\sqrt{\frac{\rho(\eta_\Lambda^{x,j})}{\mu(\eta_\Lambda^{x,j})}}+\sqrt{\frac{\rho(\eta_\Lambda)}{\mu(\eta_\Lambda)}}\right)^2
        &\leq 
        2\sum_{\eta_\Lambda}\sum_{j\neq \eta_x}\mu(\eta_\Lambda)\left(\frac{\rho(\eta_\Lambda^{x,j})}{\mu(\eta_\Lambda^{x,j})}+\frac{\rho(\eta_\Lambda)}{\mu(\eta_\Lambda)}\right) 
        \\\
        &\leq 
        \sum_{\eta_\Lambda}\sum_{j\neq \eta_x}\left(\frac{\mu(\eta_x)}{\mu(j_x)}\rho(\eta_{\Lambda}^{x,j})+\rho(\eta_\Lambda)\right)
        \\\
        &\leq 2\left(\frac{1}{\delta}+q\right), 
    \end{align*}
    where we used that $\mu$ is a product measure with non-degenerate marginals and $\rho\lvert_{\calF_\Lambda}$ is a probability measure.  
\end{proof}

\subsection{Pointwise monotonicity of the $\alpha$-coefficients}
As a last step, we now prove that considering larger volumes does not decrease the $\alpha$-coefficients. This is the only place in which we really need the product structure of $\mu$.

\begin{lemma}\label{lemma:monotonicity-alpha}
    If $\mu$ is a product measure, then for any $\rho \in \calM_1(\Omega)$ and $x \in \Delta \subset \Lambda \Subset \Z^d$ it holds that 
    \begin{align*}
        0 \leq \alpha_\Delta(x,\rho) \leq \alpha_\Lambda(x, \rho). 
    \end{align*}
\end{lemma}

\begin{proof}
    The product structure of $\mu$ gives us the factorisation
    \begin{align*}
    \mu(\eta_\Lambda) = \mu(\eta_x)\mu(\eta_{\Lambda \setminus x}), 
\end{align*}
so the $\alpha$ coefficients for $x \in \Delta \subset \Lambda \Subset \Z^d$ can be written as 
\begin{align*}
    \alpha_\Delta(x,\rho) &=\sum_{\eta_\Delta}\sum_{j\neq \eta_x} \mu(\eta_x) \left[\sqrt{\frac{\rho(\eta_\Delta^{x,j})}{\mu(j_x)}}-\sqrt{\frac{\rho(\eta_\Delta)}{\mu(\eta_x)}}\right]^2,
    \\\
    \alpha_\Lambda(x,\rho) &= \sum_{\eta_\Delta}\sum_{j\neq \eta_x}\sum_{\xi_\Lambda: \xi_\Delta = \eta_\Delta}
    \mu(\eta_x)\left[\sqrt{\frac{\rho(\xi_\Lambda^{x,j})}{\mu(j_x)}}-\sqrt{\frac{\rho(\xi_\Lambda)}{\mu(\eta_x)}}\right]^2.
\end{align*}
For fixed $\eta_\Delta$ and $j\neq \eta_x$ we can write 
\begin{align*}
    \rho(\eta_\Delta) = \sum_{\xi_\Lambda: \xi_\Delta = \eta_\Delta}\rho(\xi_\Lambda), \quad \rho(\eta_\Delta^{x,j}) = \sum_{\xi_\Lambda: \xi_\Delta = \eta_\Delta}\rho(\xi_\Lambda^{x,j}). 
\end{align*}
So in order to show that 
\begin{align*}
    0 \leq \alpha_\Delta(x,\rho) \leq \alpha_\Lambda(x,\rho),
\end{align*}
it suffices to prove a subadditivity property for functions of the form
\begin{align*}
    \Phi(u,v) = \left(\sqrt{\frac{u}{c}}-\sqrt{\frac{v}{d}}\right)^2, 
\end{align*}
for $c,d>0$. For $u_1, u_2,v_1,v_2 > 0$ we have 
\begin{align*}
    &\Phi(u_1,v_1) + \Phi(u_2,v_2) - \Phi(u_1 + u_2, v_1 + v_2) 
    \\\
    %=
    %&\left(\sqrt{\frac{u_1}{c}}-\sqrt{\frac{v_1}{d}}\right)^2 +  \left(\sqrt{\frac{u_2}{c}}-\sqrt{\frac{v_2}{d}}\right)^2 
    %- \left(\sqrt{\frac{u_1+u_2}{c}}-\sqrt{\frac{v_1+v_2}{d}}\right)^2
    %\\\
    %=
    %&\frac{u_1+u_2}{c} + \frac{v_1+v_2}{d} - 2\sqrt{\frac{u_1 v_1}{cd}} - 2\sqrt{\frac{u_2v_2}{cd}} - \frac{u_1+u_2}{c} - \frac{v_1+v_2}{d} + 2\sqrt{\frac{(u_1+u_2)(v_1+v_2)}{cd}}
    %\\\
    =\ 
    &2\sqrt{\frac{(u_1+u_2)(v_1+v_2)}{cd}}- 2\sqrt{\frac{u_1 v_1}{cd}} - 2\sqrt{\frac{u_2v_2}{cd}}.
\end{align*}
Therefore, to show that $\Phi$ is subadditive, it suffices to show that 
\begin{align*}
    \sqrt{(u_1 + u_2)(v_1 +v_2)} \geq \sqrt{u_1 v_1} + \sqrt{u_2 v_2}. 
\end{align*}
First, observe that 
\begin{align*}
    (v_1 u_2 - v_2 u_1)^2 \geq 0, 
\end{align*}
thus in particular 
\begin{align*}
    (v_1 u_2)^2 + (v_2u_1)^2 \geq 2v_1u_2v_2u_1. 
\end{align*}
This implies that 
\begin{align*}
    (v_1 u_2 + v_2 u_1)^2 = (v_1 u_2)^2 + (v_2u_1)^2 + 2v_1u_2v_2u_1\geq 4v_1u_2v_2u_1. 
\end{align*}
Since both sides are non-negative, we can take square roots and obtain
\begin{align*}
     v_1 u_2 + v_2 u_1 \geq 2 \sqrt{v_1u_2v_2u_1}.
\end{align*}
This in turn implies that 
\begin{align*}
    (v_1 + v_2)(u_1 +u_2) \geq (\sqrt{v_1 u_1}+\sqrt{v_2u_2})^2. 
\end{align*}
But up to taking square roots on both sides, this is precisely what we wanted to show. 
\end{proof}

With all of these rather technical estimates in place, we are finally ready to provide the proof of Proposition \ref{proposition:time-averaged-non-reversible-holley-stroock}. 

\begin{proof}[Proof of Proposition \ref{proposition:time-averaged-non-reversible-holley-stroock}]
    By combining Lemma \ref{lemma:time-averaged-zero-loss-ineq} with the pointwise estimates in Lemma \ref{lemma:upper-bound-beta} we obtain 
    \begin{align}\label{ineq:starting-point}
        \frac{\mathbf{c}}{2}\sum_{x \in \Lambda}\int_0^T \alpha_\Lambda(x,\nu_s)ds 
    \leq 
    \left(\frac{2q}{\delta}\right)^{1/2}
    \sum_{x \in \Lambda}\gamma_\Lambda(x) \int_0^T \sqrt{\alpha_\Lambda(x,\nu_s)}ds. 
    \end{align}
    By Lemma \ref{lemma:monotonicity-alpha} and Lemma \ref{lemma:properties-rho} we have the following pointwise estimates for all $n \in \N$ and $s \in [0,T]$ 
    \begin{align}\label{ineq:pointwise-lower-bound}
        \sum_{x \in \Lambda_n} \alpha_n(x,\nu_s) 
        \geq
        \mathbf{C}_1^{-1} \sum_{x \in \Lambda_n}\alpha_n(x,\nu_s)\sum_{k=1}^n\gamma_k(x) 
        \geq 
        \mathbf{C}_1^{-1} \sum_{k=1}^n \sum_{x \in \Lambda_k}\alpha_k(x,\nu_s)\gamma_k(x). 
    \end{align}
    For $k \in \N$ we now define 
    \begin{align*}
        \delta_k := \sum_{x \in \Lambda_k}\gamma_k(x)\int_0^T \alpha_k(x,\nu_s)ds. 
    \end{align*}
    Note that by definition of $\alpha_k(\cdot,\cdot)$ and $\gamma_k(\cdot)$ we have $\delta_k\geq0$ for all $k \in \N$. By combining \eqref{ineq:starting-point} and \eqref{ineq:pointwise-lower-bound} with the Cauchy--Schwarz inequality for sums we obtain 
    \begin{align*}
        \left[\sum_{k=1}^n\delta_k\right]^2 \leq \frac{2\mathbf{C_1}^2 q}{\delta} \left(\sum_{x \in \Lambda_n}\gamma_n(x) \right)\left(\sum_{x \in \Lambda_n}\gamma_n(x)\left(\int_0^T \sqrt{\alpha_n(x,\nu_s)}ds\right)^2\right)
    \end{align*}
     Another application of the Cauchy--Schwarz inequality to the integrals on the right-hand side yields 
    \begin{align*}
        \left(\int_0^T\sqrt{\alpha_n(x,\nu_s)}ds\right)^2 \leq T \int_0^T \alpha_n(x,\nu_s)ds. 
    \end{align*}
    Combining this with Lemma \ref{lemma:properties-rho} we finally obtain 
    \begin{align*}
        \left[\sum_{k=1}^n \delta_k\right]^2 \leq \mathbf{C} \delta_n n^{d-1} 
    \end{align*}
    for some constant $\mathbf{C}>0$. If there were an index $n_0 \in \N$ such that $\delta_{n_0}>0$, then for all $n>n_0$ it would hold that 
    \begin{align*}
        \frac{1}{n^{d-1}} \leq \mathbf{C}\left[\frac{1}{\sum_{k=1}^{n-1}\delta_k}-\frac{1}{\sum_{k=1}^n\delta_k}\right]. 
    \end{align*}
    The series over the terms on the right-hand side converges because of a standard telescoping argument and monotonicity. However, for $d \in \{1,2\}$, this leads to a contradiction. Therefore, we must have $\delta_n = 0$ for all $n \in \N$ and hence by continuity $\alpha_n(x,\cdot) \equiv 0$ for all $n \in \N$ and $x \in \Lambda_n$. But this implies that $\nu_s = \mu$ for all $s \in [0,T]$ and the claim follows. 
\end{proof}

The final argument of the proof can be condensed into the following elementary analytical lemma that we first state and prove before commenting on what it entails for higher dimensions. 

\begin{lemma}
    For any $C>0$ and $d \in \N$, there is a sequence $(\delta_n)_{n \in \R}$ that satisfies 
    \begin{align}\label{ineq:growth-bound}
        \forall n \in \N: \quad \delta_n \geq C n^{-d+1}\left(\sum_{k=1}^n \delta_k \right)^2
    \end{align}
    which is not identical to zero if and only if $d\geq 3$. 
\end{lemma}

\begin{proof}
    \textit{For $d=1,2$}: Let $(\delta_n)_{n \in \N}$ be a sequence that satisfies \eqref{ineq:growth-bound}. If there is an index $m$ such that $\delta_m>0$, then for all $n > m$ we obtain 
    \begin{align*}
        n^{-d+1} \leq C^{-1}\left(\frac{1}{\sum_{k=1}^{n-1}\delta_k}-\frac{1}{\sum_{k=1}^n \delta_k}\right). 
    \end{align*}
    Now by a telescoping argument the sum over the terms on the right hand side converges to a finite value whereas the sum over the left-hand side diverges. In particular, there can be no such index $m$ where $\delta_m$ is not equal to zero and we obtain that $\delta_n \equiv 0$. 
    \newline 
    \textit{For $d\geq3$}: Since $d\geq 3$ we know that $d-1 \geq 2$, in particular the series with terms $n^{-(d-1)}$ converges. So choosing the candidate series $\hat{\delta}_n = a n^{-(d-1)}$, we see that the growth bound \eqref{ineq:growth-bound} simplifies to 
    \begin{align*}
        \frac{1}{a} \geq C \left(\sum_{k=1}^n k^{-d+1}\right)^2. 
    \end{align*}
    For fixed $C>0$, this is clearly satisfied for sufficiently small values of $a>0$. 
\end{proof}

In particular, the above strategy cannot be extended to dimensions $d\geq3$. However, one may still ask whether we have lost too much information when we applied the Cauchy--Schwarz inequality for the first time, and finer estimates could still make the strategy work in $d\geq3$. But this is also not the case. To see this, let's restrict to the case of nearest-neighbour interactions, i.e., $\gamma_\Lambda(x) \lesssim \mathbf{1}_{\partial \Lambda}(x)$, and consider the sequences $(\alpha_n(x))_{n \in \N}$ that are defined by
\begin{align*}
    \alpha_n(x) = 
    \begin{cases}
        0 \quad &\text{if } x \notin \Lambda_n, \\
        a \cdot k^{-2d+2} &\text{if } x \in (\Lambda_k \setminus \Lambda_{k-1}) \subset \Lambda_n, 
    \end{cases}
\end{align*}
where $a>0$ is some constant which will be determined later. 
Then for fixed $x \in \Z^d$, the sequence $(\alpha_n(x))_{n \in \N}$ satisfies the required monotonicity in $n$ and additionally we have 
\begin{align*}
    \sum_{x \in \Lambda_n} \alpha_n(x) \leq c(d) a \sum_{k=1}^n k^{-d+1} \quad 
    \text{and}
    \quad 
    \sum_{x \in \partial \Lambda_n}\sqrt{\alpha_n(x)} \leq c(d) \sqrt{a}, 
\end{align*}
where $c(d)$ is a dimension dependent constant. For $d\geq 3$, this tells us that by choosing $a$ sufficiently small, this non-vanishing sequence satisfies the growth bound \eqref{ineq:starting-point}.

\subsection*{Acknowledgements}
The author thanks Benedikt Jahnel for helpful discussions during the preparation of this manuscript, Aernout van Enter for pointers to the literature regarding extremal Gibbs measures, and Christof Külske and Christian Maes for comments on the possible quasilocality of stationary measures. 
The author additionally acknowledges the financial support of the Leibniz Association within the Leibniz Junior Research Group on Probabilistic Methods for Dynamic Communication Networks
as part of the Leibniz Competition.
\bibliography{refs}
\bibliographystyle{alpha}

\end{document}